%% file: main.tex
\documentclass[11pt,a4paper]{amsart}
\newcommand{\Title}{Regularity for Subfactors}%
\newcommand{\ShortTitle}{\Title}%

\newcommand{\AuthorOne}{Keshab Chandra Bakshi}%
\newcommand{\AuthorOneAddr}{%
	Department of Mathematics, Indian Institute of Technology Kanpur, Uttar Pradesh 208 016, India
}%

\newcommand{\AuthorOneEmail}{%
	 bakshi209@gmail.com, keshab@iitk.ac.in
}%

\newcommand{\AuthorTwo}{Indrajit Ghosh}%
\newcommand{\AuthorTwoAddr}{%
	Department of Mathematics, Indian Institute of Technology Kanpur, Uttar Pradesh 208 016, India
}%

\newcommand{\AuthorTwoEmail}{%
	indrajitghosh912@gmail.com, indrajitg@iitk.ac.in
}%

\newcommand{\SubjectClassText}{Primary 46L05, 47L40; Secondary 46L10}

\newcommand{\Keywords}{Regularity, Unitary Regularity, Normalisers, Subfactor, Type $II_1, II_\infty, III$ factor}

\newcommand{\pdfTitle}{\Title}
\newcommand{\pdfAuthor}{Indrajit Ghosh}
\newcommand{\pdfSubject}{Mathematics, Research paper}
\newcommand{\pdfKeywords}{\Keywords}
\newcommand{\pdfCreator}{TeXlive}
\newcommand{\pdfCreationDate}{\today}
\newcommand{\pdfColorLink}{true}
\newcommand{\pdfLinkColor}{cyan}
\newcommand{\pdfUrlColor}{blue}
\newcommand{\pdfCiteColor}{magenta}
\input{preamble}

\begin{document}
	
%%%%%%%%%%%%%%%%%%%%%%%%%%%%%%%%%%%%%%%%%%%%%%%%%%%%%%%%%%%%%%%%%%%%%%
        \title[\ShortTitle]{\MakeUppercase\Title}
	\author{\AuthorOne}
	\address[Keshab C. Bakshi]{\AuthorOneAddr}
	
	\email{\AuthorOneEmail}

    % Author Two
    \author{\AuthorTwo}
	\address[Indrajit Ghosh]{\AuthorTwoAddr}
	\email{\AuthorTwoEmail}

	%    General info
	\date{}
	\subjclass{\SubjectClassText}
	
%%%%%%%%%%%%%%%%%%%%%%%%%%%%%%%%%%%%%%%%%%%%%%%%%%%%%%%%%%%%%%%%%%%%%%
	
	\begin{abstract}
		\input{sections/abstract}
	\end{abstract}

	\maketitle%
        \thispagestyle{empty}%

    % \tableofcontents

    \section{Introduction}
    \input{sections/intro}
    
    \section{Preliminaries}
    \label{sec:prelims}
    \input{sections/prelims}

    \section{Regularity of unital Inclusions of von Neumann Algebras}
    \label{sec:reg-subfactor}
    \input{sections/block-diagonal-subalg}

    % \section{Regularity of Spin Model Subfactors}
    % \label{sec:spin-reg}
    % \input{sections/spin-regular}

    % \section{Regularity of Irreducible Inclusion of Simple $C^*$-algebras}
    % \label{sec:reg-simple}
    % \input{sections/simple_case}

    \section{Acknowledgement}
    \input{sections/acknowledgement}
    
	\medskip
	%\nocite{*}
	
	\bibliographystyle{amsalpha} % Also use `amsplain`
	\bibliography{references}

\end{document}

%% file: preamble.tex
\usepackage[top=0.9in, bottom=1in, left=0.7in, right=0.7in]{geometry}
\usepackage{amsmath, amssymb, amsthm} % amssymb internally loads amsfonts
\usepackage[utf8]{inputenc}
\usepackage[T1]{fontenc}
\usepackage{mathtools}
\usepackage{mathrsfs} % renders \mathscr cmd
\usepackage{xfrac} % renders diagonal frac notation: use \sfrac{}{}
\usepackage{dsfont} % renders '1' for characteristic function...
\usepackage{array}
\usepackage{verbatim}
\usepackage{graphicx}
\usepackage{mdframed}
\usepackage{enumitem} % Give extra customization on top of itemize and enumerate
\usepackage{hyperref}
\hypersetup{
	pdftitle={\pdfTitle},
	pdfauthor={\pdfAuthor},
	pdfsubject={\pdfSubject},
	pdfcreationdate={\pdfCreationDate},
	pdfcreator={\pdfCreator},
	pdfkeywords={\pdfKeywords},
	colorlinks=\pdfColorLink,
	linkcolor={\pdfLinkColor},
	urlcolor=\pdfUrlColor,
	citecolor=\pdfCiteColor,
	pdfpagemode=UseOutlines,
}
\usepackage{tikz-cd} % Online editor: https://tikzcd.yichuanshen.de/
\usepackage{lipsum}
\usetikzlibrary{matrix,arrows}
\usepackage[english]{babel}
\usepackage{lmodern}
\usepackage{bbm} % For typing `set of natural nums`, e.g - \mathbbm{N}
\usepackage[dvipsnames]{xcolor}
\usepackage[most]{tcolorbox}
\usepackage{xparse}

\theoremstyle{plain}
\newtheorem{theorem}{Theorem}[section]
\newtheorem{prop}[theorem]{Proposition}
\newtheorem{lem}[theorem]{Lemma}
\newtheorem{cor}[theorem]{Corollary}

\theoremstyle{definition}
\newtheorem{definition}[theorem]{Definition}
\newtheorem{example}[theorem]{Example}

\theoremstyle{remark}

\numberwithin{equation}{section}

\newtheoremstyle{ser}% name
{8pt}% Space above
{8pt}% Space below
{\it}% Body font
{}% Indent amount
{\sf}% Theorem head font
{:}% Punctuation after theorem head
{6mm}% Space after theorem head
{}% Theorem head spec (can be left empty, meaning `normal')

\theoremstyle{ser}

\newtheoremstyle{serr}% name
{8pt}% Space above
{8pt}% Space below
{\normalfont}% Body font
{}% Indent amount
{\sf}% Theorem head font
{.}% Punctuation after theorem head
{6mm}% Space after theorem head
{}% Theorem head spec (can be left empty, meaning `normal')

\theoremstyle{serr}

\theoremstyle{ser}

\theoremstyle{ser}

\newtheoremstyle{collabquestion}
  {8pt}
  {8pt}
  {\normalfont}
  {}
  {\sffamily\bfseries\color{blue!70!black}}
  {.}
  {.5em}
  {}

\theoremstyle{collabquestion}
\newtheorem{qninner}{Question}

\definecolor{indraRed}{rgb}{0.593, 0.183, 0.183}
\definecolor{indraPink}{rgb}{0.858, 0.188, 0.478}
\definecolor{indraBlue}{rgb}{0, 0.199, 0.398}
\definecolor{madridBlue}{rgb}{0.199, 0.199, 0.695}
\definecolor{metropolisThemeColor}{rgb}{0.105, 0.214, 0.234}
\definecolor{metropolisBarColor}{rgb}{0.984, 0.0.515, 0.015}
\definecolor{UBCblue}{rgb}{0.04706, 0.13725, 0.26667} % UBC Blue (primary)
\definecolor{UBCgrey}{rgb}{0.3686, 0.5255, 0.6235} % UBC Grey (secondary)

\makeatletter
\def\mathcolor#1#{\@mathcolor{#1}}
\def\@mathcolor#1#2#3{%
	\protect\leavevmode
	\begingroup
	\color#1{#2}#3%
	\endgroup
}
\makeatother
\makeatletter
\def\ps@headings{\ps@empty
  \def\@evenhead{\normalfont\scriptsize\hfil \leftmark{}{}\hfil}%
  \def\@oddhead{\normalfont\scriptsize\hfil \rightmark{}{}\hfil}%
  \let\@mkboth\markboth
  \def\@evenfoot{\normalfont\scriptsize\hfil\thepage\hfil}%
  \def\@oddfoot{\normalfont\scriptsize\hfil\thepage\hfil}%
}
\makeatother

\makeatletter
\def\author@andify{%
  \nxandlist {\unskip ,\penalty-1 \space\ignorespaces}%
    {\unskip {} \@@and~}%
    {\unskip \penalty-2 \space \@@and~}%
}
\makeatother
\input{macros}

%% file: macros.tex
\definecolor{indraRed}{rgb}{0.593, 0.183, 0.183}
\definecolor{indraPink}{rgb}{0.858, 0.188, 0.478}
\definecolor{indraBlue}{rgb}{0, 0.199, 0.398}
\definecolor{madridBlue}{rgb}{0.199, 0.199, 0.695}
\definecolor{metropolisThemeColor}{rgb}{0.105, 0.214, 0.234}
\definecolor{metropolisBarColor}{rgb}{0.984, 0.0.515, 0.015}
\definecolor{UBCblue}{rgb}{0.04706, 0.13725, 0.26667} % UBC Blue (primary)
\definecolor{UBCgrey}{rgb}{0.3686, 0.5255, 0.6235} % UBC Grey (secondary)

\newcommand{\spn}{{\operatorname{span}\,}}

\newcommand{\C}{\mathbb{C}}

\NewDocumentCommand{\mn}{m o}
{
  \mathbb{M}_{#1}\IfValueT{#2}{(#2)}
}

     \newcommand{\sB}{\mathcal B}		
     		
     \newcommand{\sD}{\mathcal D}

     \newcommand{\sG}{\mathcal G}

     \newcommand{\sM}{\mathcal M}		
     \newcommand{\sN}{\mathcal N}

     \newcommand{\sU}{\mathcal U}

\NewDocumentCommand{\nor}{g}
  {\sN\IfValueT{#1}{_{#1}}}

\NewDocumentCommand{\unor}{g}
  {\sU\sN\IfValueT{#1}{_{#1}}}

\NewDocumentCommand{\grnor}{g}
  {\sG\sN\IfValueT{#1}{_{#1}}}

\newcommand{\wkcl}[1]{\overline{#1}^{\textrm{wot}}}

\newcommand{\I}{\mathrm{I}}
\newcommand{\II}{\mathrm{II}}
\newcommand{\III}{\mathrm{III}}

\NewDocumentCommand{\hn}{m o}
{
  \mathbb{H}^u_{#1}\IfValueT{#2}{(#2)}
}

\NewDocumentCommand{\Sp}{o}{%
  \mathrm{Sp}_{\IfValueTF{#1}{#1}{n}}%
}

%% file: sections/abstract.tex
We study three naturally associated notions of regularity for a unital inclusion of von Neumann algebras \(\mathcal{B} \subseteq \mathcal{M}\): regularity, groupoid regularity, and unitary regularity. We first observe that regularity and groupoid regularity are unconditionally equivalent by showing that the standard normaliser and the groupoid normaliser of \(\mathcal{B}\) in \(\mathcal M\) generate the same linear span. The central question addressed in this work is whether regularity implies unitary regularity. For arbitrary von Neumann subalgebras, we demonstrate that this implication fails dramatically; we construct explicit examples to show that every factor (of type \(\mathrm{I}_\infty\), \(\mathrm{II}_1\), \(\mathrm{II}_\infty\), and \(\mathrm{III}\)) contains a regular von Neumann subalgebra that is not unitarily regular. In stark contrast, we prove that the behaviour changes completely in the setting of subfactors. For subfactor inclusions across all these types, regularity always implies unitary regularity, and thus all three notions of regularity coincide. 

%% file: sections/intro.tex
The study of von Neumann subalgebras and their normalisers is fundamental to the structural theory of operator algebras. Normalisers capture the symmetries of a subalgebra and play a central role in the theory of Cartan subalgebras, crossed product constructions, and subfactors. Given a unital inclusion of von Neumann algebras $\sB \subseteq \sM$, with $1_{\sB}=1_{\sM}$, there are several natural classes of elements in $\sM$ that can be regarded as normalising $\sB$. These include arbitrary normalising elements, partial isometries, and unitaries. Each class leads naturally to a corresponding notion of regularity, and understanding the relationships among these notions is the main theme of this paper.

We consider three types of normalisers of $\sB$ in $\sM$. The \emph{normaliser} of $\sB$ in $\sM$ is
\[
\nor{\sM}(\sB)
:= \{\, n \in \sM : n\sB n^* \cup n^*\sB n \subseteq \sB \,\}.
\]
The \emph{groupoid normaliser} consists of those normalising elements which are partial isometries:
\[
\grnor{\sM}(\sB)
:= \{\, v \in \nor{\sM}(\sB) : vv^*v=v \,\}.
\]
Finally, the \emph{unitary normaliser} is the group of unitaries in $\sM$ which normalise $\sB$:
\[
\unor{\sM}(\sB)
:= \{\, u \in \sU(\sM) : u\sB u^*=\sB \,\}.
\]
Thus,
\[
\unor{\sM}(\sB)
\subseteq
\grnor{\sM}(\sB)
\subseteq
\nor{\sM}(\sB).
\]
These three classes give rise to three corresponding notions of regularity. We say that $\sB$ is \emph{regular} in $\sM$ if
\[
W^*(\nor{\sM}(\sB))=\sM,
\]
\emph{groupoid regular} if
\[
W^*(\grnor{\sM}(\sB))=\sM,
\]
and \emph{unitary regular} if
\[
W^*(\unor{\sM}(\sB))=\sM.
\]

Unitary regularity is a classical and well-studied notion. The subject can be traced back to the work of Dixmier \cite{Dix1954}, who investigated unitarily regular subalgebras of a $\II_1$ factor. Building on this work, Feldman and Moore \cite{Feldman_Moore} developed the theory further by establishing a fundamental correspondence between Cartan subalgebras and measured equivalence relations. Inspired by these developments, Kumjian \cite{Kumjian_86} and Renault \cite{Renaultbook} independently introduced regularity in the $C^*$-algebraic setting, laying the foundations for the theory of Cartan subalgebras in $C^*$-algebras. Subsequently, Exel \cite{Exel2011} introduced the notion of generalized Cartan subalgebras and developed a theory of regularity for general noncommutative inclusions. The first author and Silambarasan \cite{Bakshi2026} studied unitary regularity for finite-dimensional $C^*$-algebras and obtained a characterization of unitarily regular inclusions in that setting.
In this direction, the second author and Sumit Kumar \cite{ghosh-kumar-cartan} studied generalized Cartan subalgebras of finite-dimensional von Neumann algebras and obtained a complete characterization of such subalgebras.
A natural question is whether these three notions of regularity are equivalent. For masas in von Neumann algebras, and more generally for abelian subalgebras, it is straightforward to see that they coincide (see \cite[Remark 2.2]{cartan_triples}). Beyond the abelian setting, however, the relationship between the three notions is considerably subtler. In particular, it is natural to ask whether the three notions genuinely distinguish different classes of noncommutative inclusions.

In recent work with Sumit Kumar \cite{BGK2026}, we addressed this question for finite-dimensional von Neumann algebras and completely unraveled the distinctions between regularity, groupoid regularity, and unitary regularity. In particular, we showed that these notions do not coincide in general, even for unital inclusions of finite-dimensional von Neumann algebras. This naturally leads to the broader question of determining precisely when the different notions agree and when they differ.

The main objective of this paper is to systematically determine the relationships among the three notions of regularity beyond finite dimensional case. Since every unitary is a partial isometry and every partial isometry is a normalising element, it is immediate that
\[
\text{unitary regularity}
\ \Longrightarrow\
\text{groupoid regularity}
\ \Longrightarrow\
\text{regularity}.
\]
The first question is whether the second implication can be reversed. More importantly, we ask whether regularity implies unitary regularity.

We first show that regularity and groupoid regularity are, in fact, equivalent for arbitrary unital inclusions of von Neumann algebras. The key observation is that the normalising elements are linearly generated by the groupoid normalisers. More precisely, we prove (in Theorem \ref{thm:spn-nor-equl-gr-nor})
\[
\mathrm{span}\bigl(\nor{\sM}(\sB)\bigr)
=
\mathrm{span}\bigl(\grnor{\sM}(\sB)\bigr).
\]
Consequently,
\[
W^*(\nor{\sM}(\sB))
=
W^*(\grnor{\sM}(\sB)),
\]
and hence $\sB$ is regular in $\sM$ if and only if it is groupoid regular in $\sM$.

We then turn to the relationship between regularity and unitary regularity. In contrast to the preceding result, the two notions can differ substantially for general von Neumann subalgebras. We construct explicit examples showing that regularity does not imply unitary regularity. More precisely, given a finite von Neumann algebra $\sM$ containing a unital copy of $\mn{N}$, we consider block-diagonal subalgebras associated with a decomposition
\[
N=d_1+\cdots+d_k,
\qquad k\geq 2.
\]
We show (in Theorem \ref{thm:block-diagonal-sub-fin-vNa}) that these subalgebras are always regular, whereas they are unitarily regular precisely when
\[
d_1=d_2=\cdots=d_k.
\]
This provides a systematic family of regular but non-unitarily regular subalgebras. As a consequence, every finite von Neumann algebra without an abelian direct summand---in particular, every type $\mathrm{II}_1$ factor---contains a regular subalgebra which is not unitarily regular (see Corollary \ref{cor:exists-reg-not-uni-ii-1-fact}). We further extend this phenomenon to the remaining factor types, showing that every type $\mathrm{I}_\infty$, type $\mathrm{II}_\infty$, (see Theorem \ref{thm:exists-reg-not-uni-semi-finite}) and type $\mathrm{III}$ factor admits (see Proposition \ref{prop:exists-reg-not-uni-type-iii}) a subalgebra for which regularity and unitary regularity differ.

The situation changes dramatically when one restricts attention to subfactors. In the second part of the paper, we show that the failure of regularity to imply unitary regularity described above cannot occur for subfactor inclusions. For a unital inclusion of $\mathrm{II}_1$ factors $\sB\subseteq\sM$, we prove that regularity implies unitary regularity (see Theorem \ref{thm:reg-imply-uni-ii-1-factor}). Together with the equivalence of regularity and groupoid regularity established above, this yields
\[
\text{regularity}
\iff
\text{groupoid regularity}
\iff
\text{unitary regularity}
\]
for $\mathrm{II}_1$ subfactors.

We establish analogous results for subfactors of the other types. In particular, we prove
\[
W^*(\nor{\sM}(\sB))
=
W^*(\unor{\sM}(\sB))
\]
for inclusions of type $\mathrm{III}$ factors (see Proposition \ref{prop:reg-type-iii-subfactor}), for $\sigma$-finite type $\mathrm{II}_\infty$ subfactors (see Theorem \ref{thm:type-ii-infty-case}), and for type $\mathrm{I}_\infty$ subfactors (see Theorem \ref{thm:type-i-infty-reg-equi}). Thus, although regularity and unitary regularity may differ dramatically for general von Neumann subalgebras, they coincide for subfactor inclusions across all the factor types considered here. Combining this with the general equivalence of regularity and groupoid regularity gives a complete coincidence of the three notions in the subfactor setting.

The paper is organized as follows. In \S\ref{sec:prelims}, we collect the necessary preliminaries concerning normalisers and partial isometries. We also establish the equivalence between regularity and groupoid regularity. In \S\ref{sec:reg-subfactor}, we construct examples separating regularity from unitary regularity for general von Neumann subalgebras. The equivalence of all three notions for subfactor inclusions is established in \S\ref{subsec:regularity-subfact}. 

%% file: sections/prelims.tex
% !TEX root = ../main.tex

In this section, we introduce the three notions of regularity considered in this paper. The first is \emph{unitary regularity}, in the sense of Dixmier \cite{Dix1954}; the second is \emph{groupoid regularity} \cite{Renaultbook}; and the third is \emph{regularity}, in the sense of Renault (see \cite{Renaultbook}, \cite{Renault2008CartanSI}). We begin by recalling the notion of the normaliser of a von Neumann subalgebra, which is needed to define these concepts.

Let $\sB \subset \sM$ be a unital inclusion of von Neumann algebras, where by a unital inclusion we mean that $1_\sB = 1_\sM$.

\begin{definition}
The \emph{normaliser} of $\sB$ in $\sM$ is
\[
\nor{\sM}(\sB)
:=
\{\, n \in \sM \mid n\sB n^* \cup n^*\sB n \subseteq \sB \,\}.
\]
The \emph{groupoid normaliser} of $\sB$ in $\sM$ is
\[
\grnor{\sM}(\sB)
=
\{\, v \in \nor{\sM}(\sB) \mid vv^*v = v \,\},
\]
that is, the set of partial isometries in $\nor{\sM}(\sB)$. The \emph{unitary normaliser} of $\sB$ in $\sM$ is
\[
\unor{\sM}(\sB)
=
\{\, u \in \sU(\sM) \mid u\sB u^* = \sB \,\}.
\]
\end{definition}
Note that $\unor{\sM} (\sB) = \sU(\sM) \cap \nor{\sM}(\sB) \subseteq \grnor{\sM}(\sB)$.

\begin{definition}\label{def:regularity}
Let $\mathcal{B}$ be a von Neumann subalgebra of $\mathcal{M}$. We say that
$\mathcal{B}$ is \emph{regular} in $\mathcal{M}$ if the von Neumann algebra
generated by $\nor{\mathcal{M}}(\mathcal{B})$ coincides with
$\mathcal{M}$; that is,
\[
W^*(\nor{\mathcal{M}}(\mathcal{B}))=\mathcal{M}.
\]

We say that $\mathcal{B}$ is \emph{unitary regular} if
\[
W^*(\unor{\mathcal{M}}(\mathcal{B}))=\mathcal{M},
\]
and \emph{groupoid regular} if
\[
W^*(\grnor{\mathcal{M}}(\mathcal{B}))=\mathcal{M}.
\]
\end{definition}

\begin{lem}\label{lem:spatial-iso-by-gr-nor}
    Let $\sB\subseteq \sM$ be a unital inclusion of von Neumann algebras and $v \in \grnor{\sM}(\sB)$, then its initial projection $e := v^*v$ and final projection $f := vv^*$ belong to $\mathcal{B}$. Furthermore, $v$ implements a spatial isomorphism from $e\mathcal{B}e$ to $f\mathcal{B}f$.
\end{lem}
\begin{proof}
    Because $\mathcal{B}$ is unital, $1 \in \mathcal{B}$. By the definition of the groupoid normaliser, $v\mathcal{B}v^* \subseteq \mathcal{B}$ and $v^*\mathcal{B}v \subseteq \mathcal{B}$. Therefore:$$v 1 v^* = vv^* = f \in \mathcal{B} \quad \text{and} \quad v^* 1 v = v^*v = e \in \mathcal{B}.$$For any $x \in e\mathcal{B}e$, we have $x = exe$, so $vxv^* = v(exe)v^* = f(vxv^*)f \in f\mathcal{B}f$. To see that this map is onto, let $y \in f\mathcal{B}f$. Since $v^*\mathcal{B}v \subseteq \mathcal{B}$, $v^*yv \in e\mathcal{B}e$. Multiplying by $v$ and $v^*$ gives $v(v^*yv)v^* = fyf = y$. Therefore, $v(e\mathcal{B}e)v^* = f\mathcal{B}f$
\end{proof}

\begin{prop}[{\cite[Remark 2.2]{cartan_triples}}]\label{prop:abel-equi-reg}
    Let $\mathcal{D}$ be an abelian subalgebra of the von Neumann algebra $\sM$. Then 
    \[
    \spn (\unor{\sM} (\sD)) = \spn (\grnor{\sM}(\sD)).
    \]
\end{prop}

\begin{prop}\label{prop:polar-normaliser-vna}
    Let $\sB\subset \sM$ be a unital inclusion of von Neumann algebras. If $x\in \nor{\sM}(\sB)$ has polar decomposition
\[
x=u|x|,
\]
then $u\in \grnor{\sM}(\sB)$ and $|x|\in \sB$.
\end{prop}
\begin{proof}
    If $x=0$ nothing to prove. Suppose that $x\neq 0$. Clearly, $x^*1_{\sM}x\in \sB$, so $|x|\in \sB$. Let $E_{|x|}$ be the spectral measure of $|x|$. For $0<\epsilon < \|x\|$, let $f_{\epsilon}(t)= \frac{1}{t}\ \chi_{[\epsilon, \infty]}(t)$ and $P_{\epsilon}:= E_{|x|}([\epsilon, \|x\|])$. Then, $|x|f_{\epsilon}(|x|)= P_{\epsilon}$. So, $xf_{\epsilon}(|x|)= u|x| f_{\epsilon}(|x|)= u P_{\epsilon}$. Note that, $$uP_{\epsilon}\to u \,\, \,\text{as} \,\,\,\epsilon \to 0$$ in $\sigma$-strong$*$ topology. As, $|x|\in \sB$ this implies that $f_{\epsilon}(|x|)\in \sB$. Hence $xf_{\epsilon}(|x|)\in \nor{\sM}(\sB)$. Also, note that $\|xf_{\epsilon}(|x|)\|= \|uP_{\epsilon}\|\leq 1.$ We have for any $n\in \sB$
    \[
    uP_{\epsilon}n (uP_{\epsilon})^* \to unu^*
    \]
    as $\epsilon \to 0$ in $\sigma$-strong topology. Since $uP_{\epsilon}n(uP_{\epsilon})^*\in \sB$ for all $n\in \sB$. This implies that $unu^*\in \sB$ for all $n\in \sB$. Similarly, $u^*nu\in \sB$ for all $n\in \sB$. Thus $u\in \grnor{\sM}(\sB)$.
\end{proof}

\begin{theorem}\label{thm:spn-nor-equl-gr-nor}
    Let $\sB\subseteq \sM$ be a unital inclusion of von Neumann algebras. Then 
    $$\spn \left( \nor{\sM}(\sB) \right)= \spn \left( \grnor{\sM}( \sB) \right).$$
    Consequently, $\sB$ is regular in $\sM$ if and only if it is groupoid regular in $\sM$.
\end{theorem}
\begin{proof}
    Let $x\in \nor{\sM}( \sB)$ be arbitrary. By Proposition \ref{prop:polar-normaliser-vna} $|x| \in \sB$. Since $\sB$ is the span of its unitaries, there are unitaries $u_1, \dots, u_r \in \sU(\sB)$ and scalars $\lambda_1, \dots, \lambda_r\in \C$ such that $|x| = \sum_i \lambda_i u_i$. Suppose the partial isometry in the polar decomposition of $x$ is $v$ which is in $\grnor{\sM}(\sB)$ by Proposition \ref{prop:polar-normaliser-vna}. Therefore we get, 
    \[
    x = v |x| = \sum_i \lambda_i (vu_i) \in \spn \left(\grnor{\sM}( \sB) \right).
    \]
    Thus $\nor{\sM}(\sB)  \subseteq \spn \left(\grnor{\sM}(\sB) \right)$ which implies $\spn \left( \nor{\sM}(\sB) \right) \subseteq \spn \left( \grnor{\sM}(\sB) \right).$

    On the other hand, since $\grnor{\sM} (\sB)\subseteq \nor{\sM}(\sB)$ we get $\spn \left( \grnor{\sM}( \sB) \right) \subseteq \spn \left( \nor{\sM}(\sB)\right).$
    Therefore, $$\spn \left( \nor{\sM}(\sB) \right)= \spn \left( \grnor{\sM}( \sB) \right).$$

    The last part follows from the observation:
    \[
    W^* \left( \nor{\sM}(\sB)\right)= \wkcl{\spn \left( \nor{\sM}(\sB) \right)} = \wkcl{\spn \left( \grnor{\sM}(\sB) \right)} = W^* \left( \grnor{\sM}(\sB)\right).
    \]
\end{proof}

Using Proposition \ref{prop:abel-equi-reg} and Theorem \ref{prop:polar-normaliser-vna} we get:
\begin{cor}\label{abelian case}
Let $\mathcal{D} \subseteq \mathcal{M}$ be a unital inclusion of von Neumann algebras, where $\mathcal{D}$ is abelian. Then the following statements are equivalent:
\begin{enumerate}
    \item[(i)] $\mathcal{D}$ is regular in $\mathcal{M}$.
    \item[(ii)] $\mathcal{D}$ is groupoid regular in $\mathcal{M}$.
    \item[(iii)] $\mathcal{D}$ is unitary regular in $\mathcal{M}$.
\end{enumerate}
\end{cor}

\begin{lem}\label{lem:uni_auto}
Let $\mathcal{B} \subseteq \mathcal{M}$ be a unital inclusion of von Neumann algebras. If $u \in \unor{\sM}(\sB)$, then conjugation by $u$ restricts to an automorphism of the center $\mathcal{Z}(\mathcal{B})$.
\end{lem}
\begin{proof}
    By definition, $u\mathcal{B}u^* = \mathcal{B}$. If $z \in \mathcal{Z}(\mathcal{B})$, then for any $b \in \mathcal{B}$, there exists $b' \in \mathcal{B}$ such that $b = u b' u^*$. Therefore, $(uzu^*)b = uzu^*ub'u^* = u(zb')u^* = u(b'z)u^* = ub'u^*uzu^* = b(uzu^*)$. Thus, $uzu^*$ commutes with all of $\mathcal{B}$, meaning $u\mathcal{Z}(\mathcal{B})u^* = \mathcal{Z}(\mathcal{B})$.
\end{proof}

%% file: sections/block-diagonal-subalg.tex
Let $\mathcal{M}$ be a finite von Neumann algebra. By definition, $\mathcal{M}$ possesses a unique, faithful, normal center-valued trace $T: \mathcal{M} \to \mathcal{Z}(\mathcal{M})$. This map satisfies $T(xy) = T(yx)$ for all $x,y \in \mathcal{M}$, and $T(z) = z$ for all $z \in \mathcal{Z}(\mathcal{M})$.

\begin{lem}\label{lem:tr-of-mat-unit}
    Suppose $\mathcal{M}$ contains a unital $C^*$-subalgebra isomorphic to $\mn{N}$ with standard matrix units $\{e_{ij}\}_{i,j=1}^N$ satisfying $\sum_{i=1}^N e_{ii} = 1$. Let $\mathcal{N} = \mathcal{M} \cap \{e_{ij}\}'$ be the relative commutant. Then there is a canonical spatial isomorphism $\mathcal{M} \cong \mn{N}[\mathcal{N}]$. Furthermore, the center-valued trace of each diagonal matrix unit is precisely:$$T(e_{ii}) = \frac{1}{N} 1_{\mathcal{Z}(\mathcal{M})}$$
\end{lem}
\begin{proof}
    The isomorphism $\mathcal{M} \cong \mn{N}[\mathcal{N}]$ is a standard fact, given by the map $x \mapsto \sum_{i,j=1}^N (e_{1i} x e_{j1}) \otimes E_{ij}$, where the coefficients $e_{1i} x e_{j1}$ belong to $\mathcal{N}$.
For the trace, since $e_{ij} e_{ji} = e_{ii}$ and $e_{ji} e_{ij} = e_{jj}$, the trace property yields $T(e_{ii}) = T(e_{ij}e_{ji}) = T(e_{ji}e_{ij}) = T(e_{jj})$.
Because $1 = \sum_{i=1}^N e_{ii}$ and $T$ is linear and trace-preserving, $1_{\mathcal{Z}(\mathcal{M})} = T(1) = \sum_{i=1}^N T(e_{ii}) = N \cdot T(e_{11})$. Thus, $T(e_{ii}) = \frac{1}{N} 1_{\mathcal{Z}(\mathcal{M})}$.
\end{proof}
\begin{lem}\label{lem:uni-nor-conj}
    For any unitary $u \in \mathcal{U}(\mathcal{M})$ and any $x \in \mathcal{M}$, $T(uxu^*) = T(x)$. Consequently, if $u \in \unor{\sM}(\sB)$, then $u$ acts as an automorphism of the center $\mathcal{Z}(\mathcal{B})$ that preserves the center-valued trace of its projections.
\end{lem}
\begin{proof}
    By the properties of the center-valued trace, $T(uxu^*) = T(u^*ux) = T(1x) = T(x)$. If $u$ normalises $\mathcal{B}$, $u\mathcal{B}u^* = \mathcal{B}$, so $u$ permutes the central projections of $\mathcal{B}$. If $E \in \mathcal{Z}(\mathcal{B})$, then $T(uEu^*) = T(E)$.
\end{proof}

\begin{lem}\label{lem:mn-subalg}
    Let $\mathcal{M}$ be a $\II_1$ factor. For any integer $n \geq 1$, there exists a system of matrix units $\{v_{ij}\}_{i,j=1}^n \subset \mathcal{M}$ such that $\sum_{i=1}^n v_{ii} = 1$. Furthermore, if $\mathcal{N} = \mathcal{M} \cap \{v_{ij}\}'$ is the relative commutant, then $\mathcal{N}$ is a $\II_1$ factor, $\mathcal{N} \cong v_{11}\mathcal{M}v_{11}$, and we have a canonical spatial isomorphism $\mathcal{M} \cong \mn{n}[\sN]$.
\end{lem}

\begin{proof}[Sketch of Proof]
    Since $\mathcal{M}$ is a $\II_1$ factor, its trace $\tau$ maps projections onto the continuous interval $[0,1]$. One can choose $n$ mutually orthogonal projections $p_i$ such that $\tau(p_i) = 1/n$. Because projections of equal trace are equivalent in a factor, there exist partial isometries $v_{ij}$ such that $v_{ij}^* v_{ij} = p_j$ and $v_{ij} v_{ij}^* = p_i$. These form the required matrix units, and the isomorphism $\mathcal{M} \cong \mn{n}[\sN]$ follows via the mapping $x \mapsto \sum_{i,j} v_{i1} x v_{1j} \otimes e_{ij}$.
\end{proof}

\begin{example}  \label{ex:ii-1-case}
Let $\mathcal{M}$ be any $\II_1$ factor. By Lemma \ref{lem:mn-subalg} with $n=3$, we can fix a system of $3 \times 3$ matrix units $\{v_{ij}\}_{i,j=1}^3$ and identify:
$$\mathcal{M} \cong \mn{3}[\sN]$$
where $\mathcal{N} = \mathcal{M} \cap \{v_{ij}\}'$ is the relative commutant. The unique trace on $\mathcal{M}$ takes the form $\tau(X) = \frac{1}{3} \operatorname{tr}(\tau_{\mathcal{N}}(X_{ij}))$. Define the subalgebra $\mathcal{B} \subseteq \mn{3}[\sN]$ as the algebra of block-diagonal matrices of the form:
$$\mathcal{B} = \left\{ \begin{pmatrix} x & 0 & 0 \\ 0 & y_{11} & y_{12} \\ 0 & y_{21} & y_{22} \end{pmatrix} : x, y_{ij} \in \mathcal{N} \right\}$$Note that $\mathcal{B} \cong \mathcal{N} \oplus \mn{2}[\sN]$.

\noindent\textbf{Claim.} $\mathcal{B}$ is regular in $\mathcal{M}$; that is, $W^*(\nor{\sM}(\sB)) = \mathcal{M}$.
\begin{proof}
    It suffices to exhibit a sufficient number of partial isometries in $\nor{\sM}(\sB)$ that, along with $\mathcal{B}$, generate the full matrix algebra $\mn{3}[\sN]$. Define the self-adjoint partial isometry $V \in M_3(\mathbb{C}) \subseteq \mathcal{M}$ by:$$V = \begin{pmatrix} 0 & 1 & 0 \\ 1 & 0 & 0 \\ 0 & 0 & 0 \end{pmatrix}$$Let $b \in \mathcal{B}$. Direct multiplication yields:$$V b V^* = \begin{pmatrix} 0 & 1 & 0 \\ 1 & 0 & 0 \\ 0 & 0 & 0 \end{pmatrix} \begin{pmatrix} x & 0 & 0 \\ 0 & y_{11} & y_{12} \\ 0 & y_{21} & y_{22} \end{pmatrix} \begin{pmatrix} 0 & 1 & 0 \\ 1 & 0 & 0 \\ 0 & 0 & 0 \end{pmatrix} = \begin{pmatrix} y_{11} & 0 & 0 \\ 0 & x & 0 \\ 0 & 0 & 0 \end{pmatrix}$$The resulting matrix is clearly in $\mathcal{B}$. Because $V = V^*$, it immediately follows that $V^*\mathcal{B}V \subseteq \mathcal{B}$. Hence, $V \in G\nor{\sM}(\sB)$. Similarly, define the partial isometry $W$:$$W = \begin{pmatrix} 0 & 0 & 1 \\ 0 & 0 & 0 \\ 1 & 0 & 0 \end{pmatrix}$$
Conjugating $b$ by $W$ yields:$$W b W^* = \begin{pmatrix} y_{22} & 0 & 0 \\ 0 & 0 & 0 \\ 0 & 0 & x \end{pmatrix} \in \mathcal{B}$$
Thus, $W \in \grnor{\sM}(\sB)$. To see that $V, W$, and $\mathcal{B}$ generate $\mathcal{M}$, note that the matrix unit $e_{11} \in \mathcal{B}$. We can recover off-diagonal matrix units via multiplication:$$V e_{11} = \begin{pmatrix} 0 & 1 & 0 \\ 1 & 0 & 0 \\ 0 & 0 & 0 \end{pmatrix} \begin{pmatrix} 1 & 0 & 0 \\ 0 & 0 & 0 \\ 0 & 0 & 0 \end{pmatrix} = \begin{pmatrix} 0 & 0 & 0 \\ 1 & 0 & 0 \\ 0 & 0 & 0 \end{pmatrix} = e_{21}$$Since $e_{21}$ is in the generated von Neumann algebra, so is its adjoint $e_{12}$. Similarly, $W e_{11} = e_{31}$, which provides $e_{31}$ and $e_{13}$. Because the generated algebra contains $\mathcal{N}$ in the top-left corner, and all the standard matrix units $e_{ij}$, it generates all of $\mn{3}[\sN]$. Therefore, $W^*(\nor{\sM}(\sB)) = \mathcal{M}$
\end{proof}

\noindent\textbf{Claim.} $\mathcal{B}$ is not unitarily regular in $\mathcal{M}$; in fact, $W^*(\unor{\sM}(\sB)) = \mathcal{B} \neq \mathcal{M}$.
\begin{proof}
    Let $U \in \unor{\sM}(\sB)$ be a unitary normaliser of $\mathcal{B}$. By Lemma \ref{lem:uni_auto}, conjugation by $U$ leaves the center of $\mathcal{B}$ invariant. The center $\mathcal{Z}(\mathcal{B})$ is generated by the two central projections corresponding to the block decomposition:$$E_1 = \begin{pmatrix} 1 & 0 & 0 \\ 0 & 0 & 0 \\ 0 & 0 & 0 \end{pmatrix}, \quad E_2 = \begin{pmatrix} 0 & 0 & 0 \\ 0 & 1 & 0 \\ 0 & 0 & 1 \end{pmatrix}$$
    Because $U$ normalises $\mathcal{Z}(\mathcal{B})$, either $UE_1U^* = E_1$ or $UE_1U^* = E_2$. By Lemma \ref{lem:uni-nor-conj}, traces are preserved under unitary conjugation. In $\mathcal{M}$, we calculate the traces of these projections:$$\tau(E_1) = \frac{1}{3}, \quad \tau(E_2) = \frac{2}{3}$$Since $1/3 \neq 2/3$, it is impossible for $U E_1 U^*$ to equal $E_2$. Therefore, $U$ must act trivially on the central projections:$$UE_1U^* = E_1 \quad \text{and} \quad UE_2U^* = E_2$$This implies that $U$ commutes with $E_1$ and $E_2$. Any matrix in $\mn{3}[\sN]$ that commutes with these projections must be strictly block-diagonal with respect to the $1 \times 1$ and $2 \times 2$ partition. Thus, $U$ takes the form:$$U = \begin{pmatrix} u_1 & 0 & 0 \\ 0 & U_{22} & U_{23} \\ 0 & U_{32} & U_{33} \end{pmatrix}$$where $u_1 \in \mathcal{U}(\mathcal{N})$ and the bottom-right $2 \times 2$ block is a unitary in $\mn{2}[\sN]$. However, this is precisely the definition of $\mathcal{B}$. Therefore, any unitary normaliser of $\mathcal{B}$ already belongs to $\mathcal{B}$.$$\unor{\sM}(\sB) \subseteq \mathcal{U}(\mathcal{B})$$Consequently, the von Neumann algebra generated by the unitary normaliser is just $\mathcal{B}$ itself:
    $$W^*(\unor{\sM}(\sB)) = \mathcal{B} \subsetneq \mathcal{M}$$
    
    This completes the proof that $\mathcal{B}$ is not unitarily regular.
\end{proof}
\end{example}

\begin{theorem}\label{thm:block-diagonal-sub-fin-vNa}
    Let $\mathcal{M}$ be a finite von Neumann algebra containing a unital sub-algebra isomorphic to $\mn{N}$ (for some natural number $N \geq 2$). Let $N = d_1 + d_2 + \dots + d_k$ be a partition of $N$ into positive integers ($k \geq 2$). Let $\mathcal{N} = \mathcal{M} \cap \mn{N}'$. Define the block-diagonal von Neumann subalgebra:$$\mathcal{B} = \bigoplus_{i=1}^k \mn{d_i}[\mathcal{N}] \subseteq \mn{N}[\mathcal{N}] \cong \mathcal{M}$$Then the following hold:
    \begin{itemize}
        \item[(i)] $\mathcal{B}$ is always regular in $\mathcal{M}$.
        \item[(ii)] $\mathcal{B}$ is unitarily regular in $\mathcal{M}$ if and only if $d_1 = d_2 = \dots = d_k$.
    \end{itemize}
\end{theorem}
\begin{proof}
    \textbf{(Regularity)}. We must show $W^*(\grnor{\sM}(\sB)) = \mathcal{M}$.
The partition of $N$ induces a partition of the index set $\{1, 2, \dots, N\}$ into disjoint subsets $I_1, I_2, \dots, I_k$, where $\vert{}I_m\vert{} = d_m$. The algebra $\mathcal{B}$ consists of matrices in $\mn{N}[\mathcal{N}]$ that are zero outside the diagonal blocks $I_m \times I_m$. 

For any two indices $r, s \in \{1, \dots, N\}$, define the element:$$v = e_{rs} + e_{sr},$$ here $(e_{ij})_{i,j = 1, \dots, N}$ are the standard matrix units of $\mn{N}[\sN] \cong \sM$. Since $v = v^*$ and $v^3 = v$, $v$ is a self-adjoint partial isometry.

Let $b \in \mathcal{B}$. Because $b$ is block-diagonal, its entries $b_{xy}$ are zero unless $x, y$ belong to the same block $I_m$. Conjugating $b$ by $v$ yields $v b v^*=b_{ss} e_{rr} + b_{rr} e_{ss}\in \sB$.

Thus, $v \mathcal{B} v^* \subseteq \mathcal{B}$, meaning $v \in \grnor{\sM}(\sB)$.

By multiplying $v = e_{rs} + e_{sr}$ with appropriate diagonal matrix units (which are contained in $\mathcal{B}$), we can isolate $e_{rs}$. Thus, the von Neumann algebra generated by $\grnor{\sM}(\sB)$ contains all matrix units $e_{rs}$ and the block-diagonal entries $\mathcal{N}$. Together, these generate all of $\mn{N} \cong \mathcal{M}$. Therefore, $\mathcal{B}$ is regular. 

\textbf{(Unitary regularity)}. Let $E_i = \sum_{r \in I_i} e_{rr}$ be the central projection in $\mathcal{B}$ corresponding to the identity of the $i$-th block, $\mn{d_i}[\mathcal{N}]$. The center of $\mathcal{B}$ is given by:$$\mathcal{Z}(\mathcal{B}) = \bigoplus_{m=1}^k \mathcal{Z}(\mathcal{N}) E_m$$By Lemma \ref{lem:tr-of-mat-unit}, the center-valued trace of $E_i$ is:$$T(E_i) = \sum_{r \in I_i} T(e_{rr}) = d_i \left( \frac{1}{N} 1_{\mathcal{Z}(\mathcal{M})} \right) = \frac{d_i}{N} 1_{\mathcal{Z}(\mathcal{M})}$$

$(\impliedby)$ : Assume $d_1 = d_2 = \dots = d_k = d$. Then $N = kd$, and $\mathcal{M}$ is a $k \times k$ block matrix algebra over $\mn{d}[\mathcal{N}]$, with all blocks being of equal size.
Let $P \in \mn{k} \subseteq \mathcal{M}$ be any permutation matrix acting on the $k$ blocks. Because all blocks have the same dimension $d$, $P$ is a global unitary in $\mathcal{M}$ ($P^*P = PP^* = 1$). Conjugating $\mathcal{B}$ by $P$ merely permutes the identical blocks, so $P \mathcal{B} P^* = \mathcal{B}$.
Hence, $P \in \unor{\sM}(\sB)$. These unitary block permutations, combined with the block-diagonal elements of $\mathcal{B}$, trivially generate the entire matrix algebra $\mathcal{M}$. Thus, $W^*(\unor{\sM}(\sB)) = \mathcal{M}$.

$(\implies)$ : Assume $\mathcal{B}$ is unitarily regular, so $W^*(\unor{\sM}(\sB)) = \mathcal{M}$. Suppose, for the sake of contradiction, that there exist blocks of different sizes, i.e., $d_i \neq d_j$ for some $i \neq j$.

Let $U \in \unor{\sM}(\sB)$. By Lemma \ref{lem:uni-nor-conj}, $U$ acts as a trace-preserving automorphism on $\mathcal{Z}(\mathcal{B})$. Therefore, $U E_i U^*$ must be a projection in $\mathcal{Z}(\mathcal{B})$. Because of the structure of $\mathcal{Z}(\mathcal{B})$, we can uniquely write:$$U E_i U^* = \sum_{m=1}^k p_{im} E_m$$where each $p_{im}$ is a projection in $\mathcal{Z}(\mathcal{N}) = \mathcal{Z}(\mathcal{M})$. (Since $U E_i U^*$ is a projection, $p_{im} p_{in} = 0$ for $m \neq n$). Now, apply the center-valued trace $T$ to both sides. By Lemma 2, $T(U E_i U^*) = T(E_i) = \frac{d_i}{N} 1_{\mathcal{Z}(\mathcal{M})}$. Applying $T$ to the right side gives:$$T\left( \sum_{m=1}^k p_{im} E_m \right) = \sum_{m=1}^k p_{im} T(E_m) = \sum_{m=1}^k p_{im} \frac{d_m}{N}$$Equating the two yields an identity in the abelian von Neumann algebra $\mathcal{Z}(\mathcal{M})$:$$\frac{d_i}{N} 1_{\mathcal{Z}(\mathcal{M})} = \sum_{m=1}^k p_{im} \frac{d_m}{N}$$Multiply both sides of this equation by the projection $p_{ij}$. Because the projections $p_{im}$ are mutually orthogonal, $p_{ij} p_{im} = 0$ for $m \neq j$, and $p_{ij}^2 = p_{ij}$. This leaves:$$p_{ij} \frac{d_i}{N} = p_{ij} \frac{d_j}{N}$$$$p_{ij} \left( \frac{d_i - d_j}{N} \right) = 0$$Since we assumed $d_i \neq d_j$, the scalar $\frac{d_i - d_j}{N}$ is non-zero. This forces:$$p_{ij} = 0$$This means the central projection $U E_i U^*$ contains absolutely no support on $E_j$. Consequently, no unitary $U \in \unor{\sM}(\sB)$ can map any non-zero element of block $i$ into block $j$.

Because this holds for all $U \in \unor{\sM}(\sB)$, the von Neumann algebra generated by the unitary normaliser, $W^*(\unor{\sM}(\sB))$, will have strict zeroes in the off-diagonal block corresponding to $(i, j)$. Thus, $W^*(\unor{\sM}(\sB)) \subsetneq \mathcal{M}$, contradicting our assumption that $\mathcal{B}$ is unitarily regular. Therefore, it must be that $d_1 = d_2 = \dots = d_k$.
\end{proof}

\begin{cor}
    Let $\mathcal{M}$ be a finite diffuse von Neumann algebra.. Then $\mathcal{M}$ contains a regular, non-unitarily regular subalgebra.
\end{cor}
\begin{proof}
    Because $\mathcal{M}$ is a finite diffuse algebra, the continuity of the center-valued trace allows us to partition the identity into three equivalent, mutually orthogonal projections. This provides a unital embedding of $M_3(\mathbb{C})$ into $\mathcal{M}$.
Once $\mn{3} \subseteq \mathcal{M}$ is secured, we apply the Theorem \ref{thm:block-diagonal-sub-fin-vNa} with the partition $3 = 1 + 2$. The block diagonal subalgebra over the relative commutant $\mathcal{N}$ will immediately supply the required example.
\end{proof}

\begin{cor}\label{cor:exists-reg-not-uni-ii-1-fact}
    Every $\II_1$ factor $\mathcal{M}$ contains a von Neumann subalgebra that is regular but not unitarily regular.
\end{cor}
\begin{proof}
By Lemma \ref{lem:mn-subalg}, $\mathcal{M}$ always contains a unital copy of $\mn{3}$. Choose the partition $3 = 1 + 2$ and apply the previous Theorem \ref{thm:block-diagonal-sub-fin-vNa}.
\end{proof}

\begin{cor}
    Let $X$ be a measure space and $\mathcal{M} = \mn{N}[L^\infty(X)]$. Let $N = d_1 + \dots + d_k$ be a partition. The block-diagonal algebra $\mathcal{B} = \bigoplus_{i=1}^k \mn{d_i}[L^\infty(X)]$ is always regular, and is unitarily regular if and only if $d_1 = \dots = d_k$.
\end{cor}
\begin{proof}
    Let $\mathcal{M} = \mn{N}[L^\infty(X)]$. This is trivially a finite von Neumann algebra. It contains the constant matrix algebra $\mn{N}$ unitally. The relative commutant is $\mathcal{N} = \mathcal{M} \cap \mn{N}' \cong L^\infty(X) \otimes 1_N \cong L^\infty(X)$. Now apply Theorem \ref{thm:block-diagonal-sub-fin-vNa}.
\end{proof}

\begin{theorem}\label{thm:exists-reg-not-uni-semi-finite}
    Let $\mathcal{M}$ be a Type $\I_\infty$ or Type $\II_\infty$ factor. Then there always exists a von Neumann subalgebra $\mathcal{B} \subseteq \mathcal{M}$ that is regular but not unitarily regular.
\end{theorem}
\begin{proof}
    Since $\mathcal{M}$ is semi-finite but properly infinite, we can partition the identity as $1 = E_1 + E_2$, where $E_1$ is a non-zero finite projection and $E_2$ is an infinite projection.

    Define the diagonal corner subalgebra:
    $$\mathcal{B} = E_1 \mathcal{M} E_1 \oplus E_2 \mathcal{M} E_2$$

    \textbf{Claim.} $\mathcal{B}$ is regular.

    Let $v \in \mathcal{M}$ be any partial isometry such that its initial projection $v^*v \leq E_2$ and its final projection $vv^* \leq E_1$. Let's check if $v \in \grnor{\sM}(\sB)$. Take any $b = b_1 \oplus b_2 \in \mathcal{B}$ (with $b_i \in E_i\mathcal{M}E_i$). Conjugating by $v$ gives:$$v b v^* = v b_2 v^*$$Because $vv^* \leq E_1$, this resulting operator is entirely supported in $E_1\mathcal{M}E_1$, which is inside $\mathcal{B}$. Similarly, $v^* b v = v^* b_1 v \in E_2\mathcal{M}E_2 \subseteq \mathcal{B}$. Thus, every such partial isometry belongs to $\grnor{\sM}(\sB)$. Because $\mathcal{M}$ is a factor, the span of partial isometries between $E_2$ and $E_1$ generates the entire off-diagonal corner $E_1\mathcal{M}E_2$ (and their adjoints generate $E_2\mathcal{M}E_1$). Therefore, $W^*(\grnor{\sM}(\sB)) = \mathcal{M}$.

    \textbf{Claim.} $\mathcal{B}$ is not unitary regular.
    
    Suppose $U \in \unor{\sM}(\sB)$. Because $U$ normalizes $\mathcal{B}$, it must permute its central projections, $E_1$ and $E_2$.
 If $U E_1 U^* = E_2$, it would imply that $E_1$ and $E_2$ are Murray-von Neumann equivalent. But $E_1$ is finite and $E_2$ is infinite, making equivalence impossible.
Therefore, $UE_1U^* = E_1$ and $UE_2U^* = E_2$. This means $U$ commutes with the projections, forcing $U \in \mathcal{B}$.
Thus, $W^*(\unor{\sM}(\sB)) = \mathcal{B} \neq \mathcal{M}$.
\end{proof}

\begin{prop}\label{prop:exists-reg-not-uni-type-iii}
    Let $\sM$ be a Type $\III$ factor. Then there always exists a von Neumann subalgebra $\sB \subseteq \sM$ that is regular but not unitarily regular.
\end{prop}
\begin{proof}
    Let $\sM$ be a Type $\III$ factor. Choose two orthogonal, non-zero projections $P_1$ and $P_2$ such that $P_1 + P_2 = 1$. Both $P_1$ and $P_2$ are infinite and Murray-von Neumann equivalent to $1$. Furthermore, we can partition $P_2$ into two mutually orthogonal, equivalent projections: $P_2 = f_1 + f_2$. We can then construct a unital copy of $\mn{2}$ living inside the corner $P_2 \sM P_2$, such that $f_1$ and $f_2$ act as the minimal diagonal matrix units (i.e., $f_1 = e_{11}$ and $f_2 = e_{22}$).

    Define the subalgebra $\sB = \mathbb{C} P_1 \oplus \mn{2} P_2$.

    \noindent\textbf{Claim.} $\sB$ is regular.

    Because $\sM$ is Type $\III$, the projections $P_1$ and $f_1$ are equivalent. There exists a partial isometry $v_0 \in \sM$ such that $v_0^* v_0 = P_1$ and $v_0 v_0^* = f_1$.

    Let $u \in \mathcal{U}(P_1 \sM P_1)$ be any unitary operator acting on the corner $P_1 \sM P_1$. Define $v = v_0 u$.
    Observe that $v^*v = u^* P_1 u = P_1$ and $vv^* = v_0 P_1 v_0^* = f_1$.

    Let $x \in \sB$. Because $\sB = \mathbb{C} P_1 \oplus \mn{2} P_2$, we can write $x = \lambda P_1 + y$, where $y \in \mn{2} P_2$. Then:
    $$v x v^* = v (\lambda P_1 + y) v^* = \lambda v P_1 v^* + v y v^*$$
    Since $v = f_1 v P_1$, the term $vyv^*$ evaluates to $f_1 v (P_1 y P_1) v^* f_1$. Because $y \in P_2 \sM P_2$ and $P_1 \perp P_2$, we have $P_1 y P_1 = 0$. Thus, $vyv^* = 0$, giving:
    $$v x v^* = \lambda f_1 + 0 = \lambda f_1$$
    Since $f_1$ is the $e_{11}$ matrix unit in $\mn{2} P_2 \subseteq \sB$, the element $\lambda f_1$ belongs to $\sB$. Thus $v\sB v^* \subseteq \sB$. 
    
    Similarly, we compute $v^* x v$:
    $$v^* x v = v^* (\lambda P_1 + y) v = \lambda v^* P_1 v + v^* y v$$
    Since $v^* = P_1 v^* f_1$, we have $v^* P_1 v = P_1 v^* (f_1 P_1 f_1) v P_1 = 0$. 
    For the second term, $v^* y v = v^* (f_1 y f_1) v$. Because $f_1 = e_{11}$ is a minimal projection in the matrix algebra $\mn{2}$, we have $f_1 y f_1 = \mu f_1$ for some scalar $\mu$. Thus:
    $$v^* x v = 0 + v^* (\mu f_1) v = \mu v^* f_1 v = \mu P_1$$
    Clearly, $\mu P_1 \in \mathbb{C} P_1 \subseteq \sB$, so $v^* \sB v \subseteq \sB$.

    Therefore, for every $u \in \mathcal{U}(P_1 \sM P_1)$, the partial isometry $v = v_0 u$ belongs to $\grnor{\sM}(\sB)$.

    The set of all such $v_0 u$ forms the space $v_0 \, \mathcal{U}(P_1 \sM P_1)$. Since the linear span of the unitary group is weakly dense in its von Neumann algebra, the weak closure of this subset of $\grnor{\sM}(\sB)$ is exactly $v_0 (P_1 \sM P_1) = f_1 \sM P_1$.

    By symmetry (using $f_2$), $\grnor{\sM}(\sB)$ also generates $f_2 \sM P_1$.
    Adding these together, $W^*(\grnor{\sM}(\sB))$ contains $(f_1 + f_2) \sM P_1 = P_2 \sM P_1$.
    Taking adjoints, $W^*(\grnor{\sM}(\sB))$ contains $P_1 \sM P_2$. Multiplying these off-diagonal corners generates the diagonal corners $P_1 \sM P_1$ and $P_2 \sM P_2$. 

    Thus, $W^*(\grnor{\sM}(\sB)) = \sM$. 

    \noindent\textbf{Claim.} $\sB$ is not unitary regular.

    Let $U \in \unor{\sM}(\sB)$. Since conjugation by $U$ restricts to an automorphism of the center $\mathcal{Z}(\sB) = \mathbb{C} P_1 \oplus \mathbb{C} P_2$, $U P_1 U^*$ must be either $P_1$ or $P_2$. But since $P_1 \sB P_1 \cong \mathbb{C}$ and $P_2 \sB P_2 \cong \mn{2}$, $U P_1 U^*$ cannot equal $P_2$. 

    Therefore $U P_1 U^* = P_1$ for all $U \in \unor{\sM}(\sB)$. This implies that every unitary normaliser commutes with the projection $P_1$. Therefore:
    $$W^*(\unor{\sM}(\sB)) \subseteq \{P_1\}' \cap \sM$$
    Because $P_1$ is a non-trivial projection and $\sM$ is a factor, $\sM$ does not commute with $P_1$.
    Therefore, $W^*(\unor{\sM}(\sB)) \subsetneq \sM$.
\end{proof}

\subsection{Regularity of Subfactor}\label{subsec:regularity-subfact}

\begin{lem}\label{lem:polar-gr-nor}
    Let $\sB \subseteq \sM$ be a unital inclusion of type $\II_1$-factors and $v \in \grnor{\sM}(\sB)$. Then there exists a partial isometry $w \in \mathcal{B}$ and a local unitary $u \in e\mathcal{M}e$ (where $e = v^*v$) such that $v = wu$ and $u(e\mathcal{B}e)u^* = e\mathcal{B}e$.
\end{lem}
\begin{proof}
    Since $v \in \mathcal{M}$ is a partial isometry, $\tau(v^*v) = \tau(vv^*)$, meaning $\tau(e) = \tau(f)$ where $f=vv^*$. Because $\mathcal{B}$ is itself a $\II_1$ factor, the restriction of $\tau$ to $\mathcal{B}$ is its unique trace, so $e$ and $f$ are Murray-von Neumann equivalent inside $\mathcal{B}$. Thus, there exists a partial isometry $w \in \mathcal{B}$ such that $w^*w = e$ and $ww^* = f$. Define $u = w^*v$. We verify that $u$ is a unitary in the corner $e\mathcal{M}e$:
$$u^*u = v^*ww^*v = v^*fv = e$$$$uu^* = w^*vv^*w = w^*fw = e$$
Furthermore, by Lemma \ref{lem:spatial-iso-by-gr-nor}, we get:
$$u(e\mathcal{B}e)u^* = w^*v(e\mathcal{B}e)v^*w = w^*(f\mathcal{B}f)w = w^* \sB w = w^*w \sB w^*w=e\mathcal{B}e.$$

Since $w \in \mathcal{B}$, we have factored $v = wu$ with the required properties.
\end{proof}

\begin{lem}\label{lem:cut-down}
     Let $\sB \subseteq \sM$ be a unital inclusion of type $\II_1$-factors and $e\in \sM$ be a projection. Let $u \in e\mathcal{M}e$ be a local unitary such that $u(e\mathcal{B}e)u^* = e\mathcal{B}e$. If $p \leq e$ is a projection in $\mathcal{B}$ such that $\tau(p) = \frac{1}{n}$ for some integer $n \geq 2$, then the cut-down $up \in \mathcal{N}:=W^*(\unor{\sM}(\sB))$.
\end{lem}
\begin{proof}
    Since $u$ normalises $e\mathcal{B}e$ and $p \in e\mathcal{B}e$, the element $q = upu^*$ is also a projection in $e\mathcal{B}e$. Traces are unitarily invariant, so $\tau(q) = \tau(upu^*) = \tau(p) = \frac{1}{n}$. Because $\mathcal{B}$ is a factor and $\tau(p) = \tau(q) = \frac{1}{n}$, we can partition the identity $1_{\mathcal{B}}$ into exactly $n$ copies of $p$, and also into $n$ copies of $q$. Choose mutually orthogonal projections $p_1, \dots, p_n \in \mathcal{B}$ and $q_1, \dots, q_n \in \mathcal{B}$ summing to $1$, along with partial isometries $v_i, z_i \in \mathcal{B}$ such that: $v_1 = p$, $v_i^*v_i = p$, and $v_iv_i^* = p_i$ and $z_1 = q$, $z_i^*z_i = q$, and $z_iz_i^* = q_i$. Define the global element $U \in \mathcal{M}$ by:$$U = \sum_{i=1}^n z_i u v_i^*$$
    
    By orthogonality of the range projections, $U^*U = \sum v_i u^* z_i^* z_i u v_i^* = \sum v_i p v_i^* = \sum p_i = 1$. Symmetrically, $UU^* = \sum q_i = 1$. Thus, $U$ is a unitary in $\mathcal{M}$. We check that $U \in \unor{\sM}(\sB)$. For any $b \in \mathcal{B}$:$$U b U^* = \left( \sum_{i=1}^n z_i u v_i^* \right) b \left( \sum_{j=1}^n v_j u^* z_j^* \right) = \sum_{i,j=1}^n z_i \big( u (v_i^* b v_j) u^* \big) z_j^*$$
    
    Since $v_i^* b v_j \in v_i^*v_i\sB v_j^*v_j=p\mathcal{B}p$, the inner term $u (v_i^* b v_j) u^* \in q\mathcal{B}q$. Consequently, the term $z_i (q\mathcal{B}q) z_j^*$ belongs to $q_i\mathcal{B}q_j \subseteq \mathcal{B}$. Summing these shows $UbU^* \in \mathcal{B}$, so $U \in \unor{\sM}(\sB) \subset \mathcal{N}$. Finally,
    $$q U p = q \left( \sum_{i=1}^n z_i u v_i^* \right) p = z_1 u v_1^* = q u p = up$$Because $q, p \in \mathcal{B} \subseteq \mathcal{N}$ and $U \in \mathcal{N}$, we conclude $up \in \mathcal{N}$.
\end{proof}

\begin{theorem}\label{thm:reg-imply-uni-ii-1-factor}
    Let $\mathcal{B} \subseteq \mathcal{M}$ be a unital inclusion of $\II_1$ factors. If $\mathcal{B}$ is regular in $\mathcal{M}$, then $\mathcal{B}$ is unitarily regular in $\mathcal{M}$.
\end{theorem}
\begin{proof}
    Let $\mathcal{N} = W^*(\unor{\sM}(\sB))$. It suffices to prove that $\grnor{\sM}(\sB) \subseteq \mathcal{N}$, as this immediately forces $\mathcal{M} = W^*(\grnor{\sM}(\sB)) \subseteq \mathcal{N}$. Let $v \in \grnor{\sM}(\sB)$. Let $e = v^*v$. By Lemma \ref{lem:polar-gr-nor}, we can write $v = wu$, where $w \in \mathcal{B} \subseteq \mathcal{N}$ and $u \in e\mathcal{M}e$ is a local unitary normalising $e\mathcal{B}e$. Our goal is to show $u \in \mathcal{N}$. Because $\mathcal{B}$ is a $\II_1$ factor, its trace takes all values in $[0,1]$. We can express the trace of $e$ as an infinite sum of unit fractions:$$\tau(e) = \sum_{k=1}^\infty \frac{1}{n_k}$$where each $n_k \geq 2$ is an integer. Consequently, we can select a sequence of mutually orthogonal projections $\{p_k\}_{k=1}^\infty$ in $\mathcal{B}$ such that $p_k \leq e$, $\tau(p_k) = \frac{1}{n_k}$, and $\sum_{k=1}^\infty p_k = e$. By Lemma \ref{lem:cut-down}, the cut-down $up_k$ belongs to $\mathcal{N}$ for every $k$. Summing these operators over all $k$ yields:$$u = ue = u \left( \sum_{k=1}^\infty p_k \right) = \sum_{k=1}^\infty u p_k$$Because $\mathcal{N}$ is a von Neumann algebra, it is closed in the strong operator topology. The series $\sum_{k=1}^\infty u p_k$ converges strongly to $u$, which ensures that $u \in \mathcal{N}$. Since $w \in \mathcal{N}$ and $u \in \mathcal{N}$, their product $v = wu$ also belongs to $\mathcal{N}$. Therefore, $\grnor{\sM}(\sB) \subseteq \mathcal{N}$, which implies:$$\mathcal{M} = W^*(\grnor{\sM}(\sB)) \subseteq \mathcal{N} = W^*(\unor{\sM}(\sB)) \subseteq \mathcal{M}$$Thus, $W^*(\unor{\sM}(\sB)) = \mathcal{M}$, proving that $\mathcal{B}$ is unitarily regular. 
\end{proof}

\begin{cor}\label{cor:equi-reg-type-ii-1-subfact}
    Let $\mathcal{B} \subseteq \mathcal{M}$ be a unital inclusion of $\II_1$ factors. Then the followings are equivalent:
    \begin{itemize}
        \item[(i)] $\sB$ is regular in $\sM$.
        \item[(ii)] $\sB$ is groupoid regular in $\sM$.
        \item[(iii)] $\sB$ is unitary regular in $\sM$.
    \end{itemize}
\end{cor}

\begin{prop}\label{prop:reg-type-iii-subfactor}
    Let $\mathcal{B} \subseteq \mathcal{M}$ be a unital inclusion of type $\III$ factors. Then
    \[
    W^*(\nor{\sM}(\sB))=W^*(\grnor{\sM}(\sB)) = W^*(\unor{\sM}(\sB)).
    \]
    Consequently $\sB$ is regular in $\sM$ if and only if it is unitary regular in $\sM$.
\end{prop}
\begin{proof}
    Let $v \in \grnor{\sM}(\sB)$. Then $e= v^*v, f:=vv^* \in \sB$ (by Lemma \ref{lem:spatial-iso-by-gr-nor}). Because $\mathcal{B}$ is a Type $\III$ factor, $e$ and $f$ are equivalent to $1$ inside $\mathcal{B}$. Therefore, there exist isometries $w_1, w_2 \in \mathcal{B}$ such that:
    \begin{align*}
        w_1^* w_1 = 1 \quad &\text{and} \quad w_1 w_1^* = e, \\
        w_2^* w_2 = 1 \quad &\text{and} \quad w_2 w_2^* = f
    \end{align*} 
    Define $U:= w_2^* v w_1\in \sM$. Then $U$ is a unitary:
    $$U^* U = w_1^* v^* w_2 w_2^* v w_1 = w_1^* (v^* f v) w_1 = w_1^* e w_1 = 1$$
    $$U U^* = w_2^* v w_1 w_1^* v^* w_2 = w_2^* (v e v^*) w_2 = w_2^* f w_2 = 1$$
    For any $b \in \mathcal{B}$:
    $$U b U^* = w_2^* \big( v (w_1 b w_1^*) v^* \big) w_2$$
    
    Since $w_1 \in \mathcal{B}$, the term $w_1 b w_1^* \in e\mathcal{B}e$. Because $v$ normalises $\mathcal{B}$, $v(e\mathcal{B}e)v^* = f\mathcal{B}f$ (by Lemma \ref{lem:spatial-iso-by-gr-nor}). Finally, multiplying by $w_2^*$ and $w_2$ maps this into $1\mathcal{B}1 = \mathcal{B}$. Thus, $U \in \unor{\sM}(\sB)$.
    
    Finally, we can recover $v$ using elements of $\mathcal{B}$ and $U$:$$w_2 U w_1^* = w_2 w_2^* v w_1 w_1^* = f v e = v$$

    Therefore every partial isometry $v \in \grnor{\sM}(\sB)$ is generated by $\unor{\sM}(\sB)$ and $\mathcal{B}$. Therefore, $W^*(\grnor{\sM}(\sB)) = W^*(\unor{\sM}(\sB))$.
\end{proof}

\begin{theorem}\label{thm:type-ii-infty-case}
    Let $\mathcal{M}$ be a $\sigma$-finite Type $\II_\infty$ factor, and let $\mathcal{B} \subseteq \mathcal{M}$ be a von Neumann subfactor which is also of Type $\II_\infty$.  Then
    \[
    W^*(\nor{\sM}(\sB))=W^*(\grnor{\sM}(\sB)) = W^*(\unor{\sM}(\sB)).
    \]
    Consequently $\sB$ is regular in $\sM$ if and only if it is unitary regular in $\sM$.
\end{theorem}
\begin{proof}
    Let $v \in \grnor{\sM}(\sB)$. Let $e = v^*v$ and $f = vv^*$. Then by Lemma \ref{lem:spatial-iso-by-gr-nor}, $e, f \in \sB$ and $v (e\mathcal{B}e) v^* = f\mathcal{B}f$.

    Let $\tau_{\mathcal{M}}$ be the unique faithful, normal, semi-finite trace on $\mathcal{M}$. Because $v \in \mathcal{M}$ is a partial isometry, $e \sim_{\mathcal{M}} f$, which implies $\tau_{\mathcal{M}}(e) = \tau_{\mathcal{M}}(f)$. Since $\mathcal{B}$ is also a Type $\II_\infty$ factor, the restriction $\tau_{\mathcal{B}} = \tau_{\mathcal{M}}\vert{}_{\mathcal{B}}$ acts as the unique faithful, normal, semi-finite trace on $\mathcal{B}$. Thus:
$$\tau_{\mathcal{B}}(e) = \tau_{\mathcal{B}}(f)$$
Therefore, $e$ and $f$ are equivalent inside the subfactor $\mathcal{B}$ ($e \sim_{\mathcal{B}} f$). Because $\mathcal{B}$ is a factor, any projection in $\mathcal{B}$ is either finite or infinite. We divide the proof into two exhaustive cases based on the finiteness of $e$ (and consequently $f$).

\noindent \textbf{Case I.} $e$ and $f$ are infinite projections.

In a $\sigma$-finite Type $\II_\infty$ factor, any projection with infinite trace is Murray-von Neumann equivalent to the identity projection $1_{\mathcal{B}}$. Thus, $e \sim_{\mathcal{B}} 1_{\mathcal{B}}$ and $f \sim_{\mathcal{B}} 1_{\mathcal{B}}$.

There exist isometries $w_1, w_2 \in \mathcal{B}$ such that:
$$w_1^* w_1 = 1_{\mathcal{B}}, \quad w_1 w_1^* = e$$
$$w_2^* w_2 = 1_{\mathcal{B}}, \quad w_2 w_2^* = f$$

Define the operator $U \in \mathcal{M}$ as $U = w_2^* v w_1$. Then $U$ is unitary:
$$U^* U =(w_1^* v^* w_2)(w_2^* v w_1)= w_1^* (v^*fv) w_1= w_1^* e w_1 = (w_1^* w_1) (w_1^* w_1) = 1_{\mathcal{B}} \cdot 1_{\mathcal{B}} = 1_{\mathcal{B}}$$
$$U U^* = (w_2^* v w_1) (w_1^* v^* w_2)  = w_2^* (v e v^*) w_2 = w_2^* f w_2 = w_2^* (w_2 w_2^*) w_2 = 1_{\mathcal{B}} \cdot 1_{\mathcal{B}} = 1_{\mathcal{B}}$$

Let $x \in \mathcal{B}$.
$$U x U^* = w_2^* \big( v (w_1 x w_1^*) v^* \big) w_2$$
Since $w_1 \in \mathcal{B}$, $w_1 x w_1^* \in e\mathcal{B}e$. Because $v \in \grnor{\sM}(\sB)$, $v(e\mathcal{B}e)v^* = f\mathcal{B}f$. Thus, $v(w_1 x w_1^*)v^* \in f\mathcal{B}f$. Finally, $w_2^* (f\mathcal{B}f) w_2 \subseteq \mathcal{B}$. So $U\mathcal{B}U^* \subseteq \mathcal{B}$. By symmetry, $U^*\mathcal{B}U \subseteq \mathcal{B}$, so $U$ is a unitary normalizer, that is, $U \in \unor{\sM}(\sB)$.

Note that: 
$$w_2 U w_1^* = w_2 w_2^* v w_1 w_1^* = f v e = v$$

Since $w_1, w_2 \in \mathcal{B}$ and $U \in \unor{\sM}(\sB)$, we conclude $v \in W^*(\unor{\sM}(\sB))$.

\noindent\textbf{Case II.}  $e$ and $f$ are finite projections.

Choose a sequence of mutually orthogonal projections $\{e_i\}_{i=1}^\infty \subset \mathcal{B}$ such that $e_1 = e$, $\sum_{i=1}^\infty e_i = 1_{\mathcal{B}}$, and $e_i \sim_{\mathcal{B}} e$ for all $i$. Similarly, choose $\{f_i\}_{i=1}^\infty \subset \mathcal{B}$ such that $f_1 = f$, $\sum_{i=1}^\infty f_i = 1_{\mathcal{B}}$, and $f_i \sim_{\mathcal{B}} f$ for all $i$. 
Let $u_i \in \mathcal{B}$ be partial isometries such that $u_i^* u_i = e$ and $u_i u_i^* = e_i$ (with $u_1 = e$). 
Let $z_i \in \mathcal{B}$ be partial isometries such that $z_i^* z_i = f$ and $z_i z_i^* = f_i$ (with $z_1 = f$).

Define $U \in \mathcal{M}$ as the strongly convergent sum:
$$U = \sum_{i=1}^\infty z_i v u_i^*$$
We have:
$$U^* U = \sum_{i=1}^\infty u_i v^* z_i^* z_i v u_i^* = \sum_{i=1}^\infty u_i (v^* f v) u_i^* = \sum_{i=1}^\infty u_i e u_i^* = \sum_{i=1}^\infty e_i = 1_{\mathcal{B}}$$
By identical logic, $U U^* = \sum_{i=1}^\infty f_i = 1_{\mathcal{B}}$.

Let $x \in \mathcal{B}$.
$$U x U^* = \sum_{i,j=1}^\infty z_i \big( v (u_i^* x u_j) v^* \big) z_j^*$$ 
Observe the inner term $u_i^* x u_j$. Since $u_i^* = e u_i^*$ and $u_j = u_j e$, this term strictly belongs to $e\mathcal{B}e$. Now we have $v (e\mathcal{B}e) v^* = f\mathcal{B}f$. Thus, $v(u_i^* x u_j)v^* = y_{ij}$ for some $y_{ij} \in f\mathcal{B}f$. Thus $z_i y_{ij} z_j^* \in z_i f \mathcal{B} f z_j^* = f_i \mathcal{B} f_j \subseteq \mathcal{B}$.
Because $\mathcal{B}$ is a von Neumann algebra, it is closed in the strong operator topology, meaning the infinite sum $U x U^*$ converges to an element in $\mathcal{B}$. Thus $U\mathcal{B}U^* \subseteq \mathcal{B}$. By symmetry, $U^*\mathcal{B}U \subseteq \mathcal{B}$, so $U$ is a unitary normalizer, that is, $U \in \unor{\sM}(\sB)$.

Multiplying $U$ by the local projections isolates $v$:
$$f_1 U e_1 = f \left( \sum_{i=1}^\infty z_i v u_i^* \right) e = z_1 v u_1^* = f v e = v$$
Since $e_1, f_1 \in \mathcal{B}$ and $U \in \unor{\sM}(\sB)$, we conclude $v \in W^*(\unor{\sM}(\sB))$.

Thus, $W^*(\grnor{\sM}(\sB)) \subseteq W^*(\unor{\sM}(\sB))$.
\end{proof}

\begin{theorem}\label{thm:type-i-infty-reg-equi}
    Let $\mathcal{M}$ be a Type $\I_\infty$ factor, and let $\mathcal{B} \subseteq \mathcal{M}$ be a von Neumann subfactor of Type $\I_\infty$. If $\mathcal{B}$ is regular in $\mathcal{M}$, then $\mathcal{B}$ is unitarily regular in $\mathcal{M}$.
\end{theorem}
\begin{proof}
    Let $v \in \grnor{\sM}(\sB)$ be a groupoid normalizer with initial projection $e = v^*v$ and final projection $f = vv^*$. By Lemma \ref{lem:spatial-iso-by-gr-nor}, $e, f \in \mathcal{B}$, and conjugation by $v$ implements a spatial von Neumann algebra isomorphism between the reduced corners $e\mathcal{B}e$ and $f\mathcal{B}f$.
    
    Because $\mathcal{B}$ is a Type $\I_\infty$ factor (isomorphic to $\mathcal{B}(\mathcal{H})$), the algebraic isomorphism $e\mathcal{B}e \cong f\mathcal{B}f$ implies that the projections $e$ and $f$ have the exact same rank within $\mathcal{B}$. Consequently, $e$ and $f$ are  Murray-von Neumann equivalent inside the subfactor, i.e., $e \sim_{\mathcal{B}} f$.
    
    Because $e$ and $f$ are equivalent inside $\mathcal{B}$, the remainder of the proof follows exactly as in the Type $\II_\infty$ case (see Theorem \ref{thm:type-ii-infty-case}). By separating into the cases where $e$ is an infinite projection or a finite-dimensional projection, we can construct a global normalizer $U \in \unor{\sM}(\sB)$ such that $v \in W^*(U, \mathcal{B})$. Thus, $W^*(\grnor{\sM}(\sB)) = W^*(\unor{\sM}(\sB))$.
\end{proof}

%% file: sections/acknowledgement.tex
The first author gratefully acknowledges support from grant ANRF/ECRG/2024/002328/PMS. We also extend our gratitude to Prof. Ralf Meyer for identifying an error in a previous version of this manuscript, which will be corrected in a forthcoming revision.